\documentclass[12pt]{amsart}

\usepackage{amsmath,amscd,amssymb,latexsym, mathrsfs}
\usepackage[normalem]{ulem} 
\usepackage{booktabs}
\usepackage{tikz}

\theoremstyle{plain}
\newtheorem{thm}{Theorem}[section] 
\newtheorem*{thm*}{Theorem}

\newtheorem*{leem*}{Lemma}
\newtheorem{cor}[thm]{Corollary}
\newtheorem*{cor*}{Corollary}
\newtheorem{prop}[thm]{Proposition}
\newtheorem*{prop*}{Proposition}

\newtheorem*{rmk}{Remark}

\theoremstyle{definition}

\newtheorem*{def*}{Definition}

\newtheorem*{ex*}{Example}

\newcommand{\ibe}{\begin{eqnarray}}
	\newcommand{\iee}{\end{eqnarray}}
\newcommand{\be}{\begin{eqnarray*}}
	\newcommand{\ee}{\end{eqnarray*}}

\numberwithin{equation}{section}
\numberwithin{thm}{section}

\makeatletter
\@namedef{subjclassname@2020}{%
	\textup{2020} Mathematics Subject Classification}
\makeatother

\newcommand{\cancel}[1]{\ifmmode\text{\sout{\ensuremath{#1}}}\else\sout{#1}\fi}

\begin{document}


\title[Moments of Partition Statistics]{Moments of Partition Statistics, Bell polynomials and Eisenstein-type Series}


\author[S.-Y. Kang]{Soon-Yi Kang}
\address{Department of Mathematics, Kangwon National University, Chuncheon, 200-701, Republic of Korea}
\email{sy2kang@kangwon.ac.kr}
	
\author[B. Kim]{Byungchan Kim}
\address{School of Natural Sciences, Seoul National University of Science and Technology, Seoul 01811, Republic of  Korea}
\email{bkim4@seoultech.ac.kr}
	
\author[C. Lee]{Chanhee Lee}
\address{Department of Mathematics, Kangwon National University, Chuncheon, 200-701, Republic of Korea}
\email{dxscf156@kangwon.ac.kr}


\makeatletter
\@namedef{subjclassname@2020}{\textup{2020} Mathematics Subject Classification}
\makeatother

\date{\today}


\begin{abstract}
We develop a systematic method to express generating functions for moments of combinatorial statistics in terms of partition traces.
We employ an algebraic approach based on the complete Bell polynomials and their inversion formula,
alongside an analytic approach via Fa\`a di Bruno's formula.
Our approach can be applied to a wide class of combinatorial statistics,
such as the largest part of an integer partition,
the partition crank and rank, and the unimodal sequence rank.
\end{abstract}

\maketitle
	
\section{Introduction}
A \emph{partition} of a positive integer $n$ is a nonincreasing sequence of positive integers
\[
	\lambda = (\lambda_1,\lambda_2,\dots,\lambda_m)
\]
such that
\[
	\lambda_1+\lambda_2+\cdots+\lambda_m = n.
\]
We write $\lambda \vdash n$ to indicate that $\lambda$ is a partition of $n$, and denote by $p(n)$ the number of partitions of $n$.

The generating function for the partition function is
\[
\sum_{n\ge0} p(n)q^n
=
\prod_{m\ge1}\frac{1}{1-q^m}
=
\frac{q^{1/24}}{\eta(\tau)},
\qquad
\left(q=e^{2\pi i\tau},\ \tau\in\mathbb H:=\{w\in\mathbb C:\text{Im}(w)>0\}\right),
\]
where
\[
\eta(\tau):=q^{1/24}\prod_{m\ge1}(1-q^m) 
\]
is the Dedekind eta function. Thus, the partition generating function is naturally connected with modular forms, since $\eta(\tau)$ is a fundamental modular object of weight $1/2$ (with multiplier), and eta-quotients play a central role in the theory.

To investigate the arithmetic and probabilistic properties of the partition function, several partition statistics have been introduced.  Two of the most prominent statistics on partitions are the \emph{rank} and the \emph{crank}. The rank of a partition, introduced by Dyson \cite{dyson} to give a combinatorial explanation of Ramanujan's partition congruences, is defined by
\[
\operatorname{rank}(\lambda):=\lambda_1-\ell(\lambda),
\]
where $\lambda_1$ is the largest part of $\lambda$ and $\ell(\lambda)$ is the number of parts of $\lambda$.

The crank, anticipated by Dyson and later defined by Andrews and Garvan \cite{AG1988}, is defined as follows. Let $\mu(\lambda)$ denote the number of ones in $\lambda$, and let $\nu(\lambda)$ denote the number of parts of $\lambda$ that are strictly larger than $\mu(\lambda)$. Then
\[
\operatorname{crank}(\lambda):=
\begin{cases}
\lambda_1, & \text{if } \mu(\lambda)=0,\\[6pt]
\nu(\lambda)-\mu(\lambda), & \text{if } \mu(\lambda)>0.
\end{cases}
\]

For $m\in\mathbb{Z}$ and $n\geq 0$, let $M(m,n)$ and $N(m,n)$ denote the number of partitions of $n$ with crank $m$ and rank $m$, respectively.  Then the corresponding two-variable crank generating function is given by \cite{AG1988, Garvan}
\begin{equation}
	C(\zeta; q) := \sum_{\substack{m \in \mathbb{Z} \\ n \geq 0}} M(m,n)\zeta^m q^n 
	= \frac{(q)_\infty}{(\zeta q)_\infty (\zeta^{-1} q)_\infty}, 
\label{eq:intro-crank-gf}
\end{equation}
and the rank generating function is given by \cite{atkin-swinner}
\begin{equation}
	R(\zeta;q) := \sum_{\substack{m \in \mathbb{Z} \\ n \geq 0}} N(m,n)\zeta^m q^n 
	= \sum_{n \geq 0} \frac{q^{n^2}}{(\zeta q)_n (\zeta^{-1} q)_n}, \label{eq:intro-rank-gf}
\end{equation}
where here and in the sequel, $(a)_n = \prod_{k=1}^{n} (1-aq^{k-1})$ for $n \in \mathbb{N}_0 \cup \{ \infty \}$. More precisely, viewed as functions of the elliptic variable $z$ and the modular variable $\tau$, with $\zeta=e^{2\pi i z}$ and $q=e^{2\pi i\tau}$, the crank generating function $C(\zeta;q)$ is a meromorphic Jacobi form of weight $\frac12$ and index $-\frac12$, whereas the rank generating function $R(\zeta;q)$ is a mock Jacobi form of weight $\frac12$ and index $-\frac32$.
At the specialization \(\zeta=-1\), the crank generating function reduces to
\[
C(-1;q)=\frac{(q)_\infty}{(-q)_\infty^2}
= q^{1/24}\frac{\eta(\tau)^3}{\eta(2\tau)^2},
\]
and hence \(q^{-1/24}C(-1;q)\) is a modular form, whereas
\[
R(-1;q)=\sum_{n\ge0}\frac{q^{n^2}}{(-q)_n^2}=f(q)
\]
is Ramanujan's third order mock theta function.  For further details on these automorphic properties, see \cite{BringmannMahlburgRhoades2014,Kang2009,Rhoades2013}.

The \(k\)-th moment generating functions for the crank and rank are defined by
\begin{equation}
C_k(q) := \sum_{n \geq 0} \sum_{m \in \mathbb{Z}} m^k M(m,n) q^n, \label{eq:intro-crank-moment}
\end{equation}
and
\begin{equation}
R_k(q) := \sum_{n \geq 0} \sum_{m \in \mathbb{Z}} m^k N(m,n) q^n. \label{eq:intro-rank-moment}
\end{equation}
Equivalently,
\begin{equation}
\left.\left(\zeta\frac{d}{d\zeta}\right)^k C(\zeta;q)\right|_{\zeta=1}=C_k(q),
\qquad
\left.\left(\zeta\frac{d}{d\zeta}\right)^k R(\zeta;q)\right|_{\zeta=1}=R_k(q). \label{eq:intro-moment-operator}
\end{equation}
Atkin and Garvan \cite{AtkinGarvan2003} showed that the crank moment generating functions are quasi-modular forms, whereas Bringmann, Garvan, and Mahlburg \cite{BringmannGarvanMahlburg2009} proved that the rank moment generating functions are quasi-mock modular forms. See also \cite{BringmannMahlburgRhoades2014,Rhoades2013} for interpretations of these moment generating functions in terms of Jacobi and mock Jacobi forms.
Moreover, taking the logarithmic derivative of the partition generating function naturally produces divisor sums: 
\[
	q\frac{d}{dq} \log \left( \frac{1}{(q)_\infty} \right)
	=
	\sum_{n\ge1}\sigma_1(n)q^n,
\]
where, for $j\ge0$,
\[
\sigma_j(n):=\sum_{d\mid n} d^j.
\]
These divisor sums appear in the Fourier expansions of Eisenstein series. For even integers $k\ge2$, the normalized Eisenstein series is defined by
\[
E_k(\tau):=1-\frac{2k}{B_k}\sum_{n\ge1}\sigma_{k-1}(n)q^n,
\]
where $B_k$ denotes the $k$-th Bernoulli number. 

However, unlike the Eisenstein series of weight at least $4$, which are modular forms and furnish the basic building blocks for spaces of modular forms on $\mathrm{SL}_2(\mathbb Z)$, the function $E_2$ is not itself a modular form on $\mathrm{SL}_2(\mathbb Z)$. This failure of modularity is corrected by adding a non-holomorphic term to form the completed Eisenstein series
\[
E_2^*(\tau):=E_2(\tau)-\frac{3}{\pi\,\text{Im}(\tau)},
\]
which transforms as a modular form of weight $2$ on $\mathrm{SL}_2(\mathbb Z)$. Surprisingly, the crank and rank moments are deeply connected with (mock) Eisenstein series.

Recently, Amdeberhan, Griffin, Ono, and Singh \cite[Theorem 1.2]{TracepartitionOno2025} proved that the even crank moments admit explicit expressions in terms of partition Eisenstein traces. To describe these, we write a partition in frequency notation as
\[
\lambda=(1^{m_1},2^{m_2},\dots,k^{m_k})\vdash k,
\]
where \(m_j\) denotes the multiplicity of the part \(j\), and
\[
\ell(\lambda)=m_1+m_2+\cdots+m_k.
\]
For each partition $\lambda=(1^{m_1},2^{m_2},\ldots,k^{m_k})$, define the monomial
\[
X_\lambda:=\prod_{j=1}^k X_j^{m_j}.
\]
Then, for a function \(\phi:\mathcal P\to\mathbb C\) on the set \(\mathcal P\) of all partitions, the \emph{partition trace} associated with \(\phi\) is defined by
\begin{equation}
\mathrm{Tr}_k(\phi;X_1,\ldots,X_k):=\sum_{\lambda\vdash k}\phi(\lambda)X_\lambda.
\label{eq:intro-partition-trace}
\end{equation}

Furthermore, \emph{partition Eisenstein traces} are the partition traces associated with the sequence \(G=\{G_k\}_{k\geq 1}\), where
\begin{equation}
G_{2k}(\tau):=-\frac{B_{2k}}{4k}E_{2k}(\tau)
=-\frac{B_{2k}}{4k}+\sum_{n\geq 1}\sigma_{2k-1}(n)q^n,
\label{eq:intro-Eisenstein}
\end{equation}
and \(G_{2k-1}(\tau)=0\).

We are now ready to state the crank moment formula in terms of partition Eisenstein traces. We present it in the form given in \cite[Theorem~1.1]{bringmannMock}, which is equivalent to \cite[Theorem~1.2]{TracepartitionOno2025}.

\begin{thm}\cite[Theorem~1.2]{TracepartitionOno2025}; see also \cite[Theorem~1.1]{bringmannMock}
\begin{equation}
\sum_{k \geq 0} C_k(q)\frac{z^k}{k!}
=
\frac{2\sinh(z/2)}{z(q)_\infty}
\sum_{k\geq 0}\mathrm{Tr}_k(\phi;G_1,G_2,\ldots,G_k)\,z^k,
\label{eq:intro-crank-trace}
\end{equation}
where
\[
\phi(\lambda):=\prod_{j\geq 1}\frac{2^{m_j}}{m_j!\,j!^{m_j}}
\qquad
\text{for }
\lambda=(1^{m_1},2^{m_2},\ldots,k^{m_k})\vdash k.
\]
\end{thm}

Motivated by this identity, Bringmann, Pandey, and van Ittersum established in \cite[Lemma~2.2]{bringmannMock} the following exponential expansion of the crank generating function in terms of Eisenstein series,
\begin{equation}
C(\zeta;q)
=
\frac{\sin(\pi z)}{\pi z(q)_\infty}
\exp\left(
2\sum_{k\geq 2}G_k(\tau)\frac{(2\pi i z)^k}{k!}
\right).
\label{eq:intro-crank-exp}
\end{equation}
They then defined a family of mock Eisenstein series \(f_k\) through the expansion
\begin{equation}
R(\zeta;q)
:=
\frac{\sin(\pi z)}{\pi z(q)_\infty}
\exp\left(2\sum_{k \geq 1} f_k(\tau) \frac{(2\pi i z)^k}{k!}\right),
\label{eq:intro-mock-eis}
\end{equation}
and proved the following analogous identity for the rank moments.

\begin{thm}\cite[Theorem~1.2]{bringmannMock}
\begin{equation}
	\sum_{k \geq 0} R_k (q)\frac{z^k}{k!}
	=
	\frac{2 \sinh(z/2)}{z (q)_\infty}
	\sum_{k \geq 0} \mathrm{Tr}_k(\phi; f_1,f_2,\ldots,f_k)\, z^k.
\label{eq:intro-rank-trace}
\end{equation}
\end{thm}

Recently, Matsusaka \cite{matsusaka2025applications} showed that several partition-theoretic generating functions, including the theta quotients from Ramanujan's lost notebook, MacMahon's partition functions, and reciprocal sums of parts in partitions, can be expressed uniformly in terms of complete Bell polynomials, or equivalently partition traces. From this perspective, the exponential factors in \eqref{eq:intro-crank-exp} and \eqref{eq:intro-mock-eis}, apart from the common prefactor \(\sin(\pi z)/(\pi z(q)_\infty)\), are precisely of the form of the generating function for complete Bell polynomials given in \eqref{eq:com-bell}. This suggests that Matsusaka's approach, based on Fa\`a di Bruno's formula and the associated Bell polynomials, also provides a natural way to understand the partition trace identities for the crank and rank generating functions and their moments.

Motivated by this observation, we study in this paper the relations among crank moments, rank moments, Eisenstein series, and mock Eisenstein series in a unified manner.
More precisely, we show not only that crank and rank moments can be expressed in terms of Eisenstein series and mock Eisenstein series, but also that these Eisenstein-type series can be recovered explicitly from the moment functions as partition traces. In particular, the mock Eisenstein series defined implicitly through the exponential expansion of the rank generating function, admit explicit descriptions in terms of rank moments. As a consequence, the functions \((q)_\infty C_{2k}(q)\) generate the same
algebra as the Eisenstein series, namely the algebra of quasi-modular forms.
On the rank side, the algebra generated by \((q)_\infty R_{2k}(q)\) together
with the Eisenstein series agrees with the algebra generated by the mock
Eisenstein series together with the Eisenstein series.

Before stating our main result, we need to introduce some terminology. Suppose that the statistic \(s(\lambda) \)  is an integer-valued function on combinatorial objects such as integer partitions. We let $S(n)$ be the number of combinatorial objects of weight $n$ and $S(m,n)$ denote  the number of objects $\lambda$ of weight $n$ with $s(\lambda)=m$. We further assume that $s(\lambda)$ is symmetric,  in the sense that 
\begin{equation}\label{eq:stat_symmetry}
S(m,n) = S(-m,n).
\end{equation}
Let $G_S (\zeta; q)$
 be the generating function for objects with statistic \(s\), i.e.,
\[
	G_S (\zeta ; q) = \sum_{ \substack{n \ge 0 \\ m \in \mathbb{Z} }} S(m,n) \zeta^m q^n.
\] 
Assume that
\begin{equation}
	G_S (\zeta;q)
	=
	\frac{\sin(\pi z)}{\pi z} G_S (1,q)
	\exp\left(
	2\sum_{n\geq 2}H_n(\tau)
	\frac{(2\pi iz)^n}{n!}
	\right).
\label{eq:intro-H-exp}
\end{equation}

Let $M_r (q)$ be the $r$-th moment generating function for $S(m,n)$, i.e.,
\[
	M_r (q) := \sum_{n \ge 0} \left( \sum_{m \in \mathbb{Z}} m^r S(m,n) \right) q^n
\]
and normalized $k$-th moment
\[
	\mathcal{M}_r (q) = \frac{M_r(q)}{G_S (1;q)}.
\]
Now we are ready to state our main result.

\begin{thm} \label{thm:main}
For positive integers $k$, let $\mathcal{M}_{2k}(q)$ 
be the normalized $2k$-th moment generating function 
for the statistic with the symmetry \eqref{eq:stat_symmetry}. Then,
\begin{align*}
	\mathcal{M}_{2k} (q) &=  \mathrm{Tr}_{k}\left( \phi_{M}; H_2 + \frac{B_2}{4}, H_4 + \frac{B_4}{8},\ldots, H_{2k} + \frac{B_{2k}}{4k} \right), \\
	H_{2k} &= - \frac{B_{2k}}{4k} + \mathrm{Tr}_{k}\left(\phi_{S} ; \mathcal{M}_2(q), \mathcal{M}_4(q), \ldots, \mathcal{M}_{2k}(q)\right),
\end{align*}
where, for $\lambda=(1^{m_1},2^{m_2},\ldots,k^{m_k})\vdash k$,
\begin{align}
	\label{eq:phi_M}
	\phi_{M}(\lambda)
	&=
	(2k)!\prod_{j=1}^k
	\frac{2^{m_j}}{m_j! ((2j)!)^{m_j}}, \\
	\label{eq:phi_S}
	\phi_{S}(\lambda)
	&=
	\frac{(2k)!}{2}
	(-1)^{\ell(\lambda)-1}(\ell(\lambda)-1)! 
	\prod_{j=1}^k \frac{1}{m_j!((2j)!)^{m_j}}.
\end{align}
\end{thm}

This theorem provides a uniform way to treat moments of symmetric combinatorial statistics.
In later sections, we give several examples including the partition crank and rank, and the unimodal rank. When the symmetry under \(\zeta\mapsto \zeta^{-1}\) is absent, a similar approach
can still be applied, and this more general situation will be considered in Section~\ref{sec:gen_stat}.

We give two proofs of Theorem~\ref{thm:main}. The first proof is based on properties of the complete Bell polynomials and their inversion formula. This approach is short and algebraic, and it makes the inversion between the
moments and the functions \(H_{k}\) transparent. The second proof is analytic
and uses Fa\`a di Bruno's formula. This proof provides an alternative way to compute the functions \(H_k(\tau)\), especially when a direct algebraic manipulation is not immediate.

The rest of the paper is organized as follows. In Section~\ref{sec:prelim}, we recall the definition and basic properties of the complete Bell polynomials and introduce Fa\`a di Bruno's formula in terms of partition traces. In Section~\ref{sec:gen_stat}, we investigate the role of the complete Bell polynomials in
the algebraic treatment of moment generating functions. In Section~\ref{sec:main}, we give two proofs of Theorem~\ref{thm:main} and derive its corollaries for the partition crank and rank, and the unimodal rank. In Section~\ref{sec:app}, we discuss applications of our results. Finally, in Section~\ref{sec:con_rmk}, we conclude with possible directions for further
study.

\section{Preliminaries}\label{sec:prelim}

In this section, we recall the definition of the complete Bell polynomials and their inversion formula. We also introduce Fa\`a di Bruno's formula in the partition trace form used  by Matsusaka \cite{matsusaka2025applications}.

\subsection{Bell polynomials and M\"obius inversion}
We shall use set partitions.   
A \emph{set partition} of a finite set \(S\) is a collection
\[
\pi=\{B_1,B_2,\ldots,B_r\}
\]
of pairwise disjoint nonempty subsets \(B_i\subseteq S\) such that
\[
B_1\sqcup B_2\sqcup\cdots\sqcup B_r=S.
\]
The subsets \(B_i\) are called the \emph{blocks} of \(\pi\), and we write
\[
|\pi|:=r
\]
for the number of blocks of \(\pi\). We denote by \(\Pi_n\) the set of all set partitions of
\[
[n]:=\{1,2,\ldots,n\}.
\]

For a set partition \(\pi\in\Pi_n\), define
\[
\lambda(\pi):=(1^{m_1},2^{m_2},\ldots,n^{m_n})\vdash n,
\]
where \(m_j\) denotes the number of blocks of size \(j\). Thus
\[
|\pi|
=
m_1+m_2+\cdots+m_n=\ell(\lambda).
\]
For example, the set partition
\[
\pi=\{\{1,3\},\{2,4,5\}\}
\]
has associated integer partition
\[
\lambda(\pi)=(2,3)=(2^1,3^1)\vdash 5.
\]

 The Bell number $B(n)$ counts the number of set partitions of a set of $n$ elements.
Its exponential generating function is
\[
	\sum_{n=0}^\infty B(n)\frac{t^n}{n!}=\exp(e^t-1).
\]
The complete Bell polynomial $B_n(X_1,\ldots,X_n)\in \mathbb{Z}[X_1,\ldots,X_{n}]$ generalizes the Bell number and is defined by the generating function
\begin{equation}
	\sum_{n=0}^\infty B_n(X_1,\ldots,X_n)\frac{t^n}{n!}	
	=\exp\left(\sum_{j=1}^\infty X_j\frac{t^j}{j!}\right).
\label{eq:com-bell}
\end{equation}
In particular,  $B(n)=B_n(1,1,\dots,1)$. From the definition, or from the exponential formula, we have 
\[
	B_n(X_1,\ldots,X_n)
	=
	\sum_{\pi\in\Pi_n}
	\prod_{B\in\pi}X_{|B|}
	=
	\sum_{\lambda\vdash n}
	\frac{n!}{\prod_{j=1}^n m_j!(j!)^{m_j}}
	X_\lambda.
\]
For $\lambda=(1^{m_1},\ldots,n^{m_n})\vdash n$, define
\begin{equation}
	\phi_B(\lambda)
	=
	n!\prod_{j=1}^n \frac{1}{m_j!(j!)^{m_j}}.
\label{eq:phiB}
\end{equation}
Thus \(\phi_B(\lambda)\) is the number of set partitions of \([n]\) whose
block-size pattern is \(\lambda\).

We also recall the inverse relation for the complete Bell polynomials. Suppose that
two sequences \(\{X_n\}_{n\ge1}\) and \(\{Y_n\}_{n\ge1}\) are related by
\[
1+\sum_{n\ge1}Y_n\frac{t^n}{n!}
=
\exp\left(\sum_{n\ge1}X_n\frac{t^n}{n!}\right).
\]
In other words,
\[
Y_n=B_n(X_1,\ldots,X_n)=\sum_{\lambda\vdash n}\phi_B(\lambda)X_\lambda.
\]
Taking the logarithm of both sides gives
\[
\sum_{n\ge1}X_n\frac{t^n}{n!}
=
\log\left(
1+\sum_{n\ge1}Y_n\frac{t^n}{n!}
\right).
\]
Hence
\begin{equation}
X_n
=
\sum_{\pi\in\Pi_n}
(-1)^{|\pi|-1}(|\pi|-1)!
\prod_{B\in\pi}Y_{|B|}.
\label{eq:bell-inversion-set}
\end{equation}
In terms of integer partitions, this becomes
\begin{equation}
X_n
=
\sum_{\lambda\vdash n}
\mu(\lambda) \phi_B(\lambda)Y_\lambda,
\label{eq:bell-inversion-int}
\end{equation}
where 
\begin{equation}
\mu(\lambda)
:=
(-1)^{\ell(\lambda)-1}(\ell(\lambda)-1)!.
\label{eq:mobius}
\end{equation}
In probability theory, \(Y_n\) and \(X_n\) correspond to moments and
cumulants, respectively, and \eqref{eq:bell-inversion-int} is the
classical moment--cumulant inversion formula. It may also be viewed as Möbius inversion on the partition lattice; see
Stanley~\cite[Chapter~3]{StanleyEC1} for the latter viewpoint and Speed~\cite{Speed1983} for the former.

\subsection{Fa\`a di Bruno's formula and partition traces}
We recall Fa\`a di Bruno's formula in the partition trace form formulated by Matsusaka \cite{matsusaka2025applications}. 
With the notation of the preceding subsection, the complete Bell polynomial can be written as
\begin{equation}
B_k(X_1,\ldots,X_k)
=
\mathrm{Tr}_k(\phi_B;X_1,\ldots,X_k),
\label{eq:bell-trace}
\end{equation}
where $\phi_B (\lambda)$ was defined by \eqref{eq:phiB}. We  now state Fa\`a di Bruno's formula in this notation. 

\begin{thm}[Fa\`a di Bruno's formula, {\cite[Theorem 1.1]{matsusaka2025applications}}]
Let $k$ be a positive integer, and let $f(x)$ and $g(x)$ be functions with all necessary derivatives. For \( \lambda=(1^{m_1},\ldots,k^{m_k})\vdash k\),
set
\begin{equation}
\phi_F(\lambda)
=
k! f^{(\ell(\lambda))}(g(x))
\prod_{j=1}^k \frac{1}{m_j!(j!)^{m_j}},
\label{eq:phiF}
\end{equation}
where $F=f\circ g$.  Then
\begin{equation}
\frac{d^k}{dx^k}f(g(x))
=
\mathrm{Tr}_k\bigl(\phi_F;g^{(1)}(x),g^{(2)}(x),\ldots,g^{(k)}(x)\bigr).
\label{eq:Faa}
\end{equation}
In particular, if $h(x)$ is another function with all necessary derivatives defined, then by taking
$f(x)=\exp(x)$ and $g(x)=\log h(x)$, we obtain
\begin{align}
\frac{h^{(k)}(x)}{h(x)}
&=
B_k\bigl((\log h)^{(1)}(x),(\log h)^{(2)}(x),\ldots,(\log h)^{(k)}(x)\bigr) \nonumber \\
&=\mathrm{Tr}_k\bigl(\phi_B;(\log h)^{(1)}(x),(\log h)^{(2)}(x),\ldots,(\log h)^{(k)}(x)\bigr). 
\label{eq:Faa-exp-log}
\end{align}
\end{thm}



\section{Moments of Combinatorial Statistics} \label{sec:gen_stat}

In this section, we consider moments of combinatorial statistics without assuming the symmetry \eqref{eq:stat_symmetry}.  In particular, we demonstrate how moment generating functions can be algebraically expressed in terms of complete Bell polynomials.

\subsection{Moment generating functions} \label{subsec:gen_stat}

We retain the notation from the introduction:
\[
	G_S(\zeta;q)=
	\sum_{\substack{m\in\mathbb Z\\ n\geq 0}} S(m,n)\zeta^m q^n,
	\qquad
	\zeta=e^{2\pi iz}.
\]
In this section, however, we do not assume that the statistic \( s(\lambda) \) is symmetric.
Note that
\[
	G_S (1;q)=\sum_{n\geq 0} S(n)q^n.
\]

For \(r\geq 0\), we define the \(r\)-th moment generating function by
\begin{align*}
	M_r(q)
	:=
	\sum_{n\geq 0}\sum_{m\in \mathbb{Z}} m^rS(m,n)q^n
	&=
	\left.
	\left(\zeta\frac{\partial}{\partial \zeta}\right)^r
	G_S(\zeta;q)
	\right|_{\zeta=1} \\
	&=
	\left.
	\frac{1}{(2\pi i)^r}
	\frac{\partial^r}{\partial z^r}
	G_S(e^{2\pi iz};q)
	\right|_{z=0}.
\end{align*}
Since
\[
	\zeta^m=e^{2\pi imz}
	=
	\sum_{r \geq 0}
	m^r \frac{(2\pi iz)^r}{r!},
\]
we have
\begin{equation}
G_S(\zeta;q)
=
\sum_{r\geq 0}
M_r(q)\frac{(2\pi iz)^r}{r!}.
\label{eq:stat-moment}
\end{equation}

Suppose that \(g_n (\tau)\) is defined by
\begin{equation}
	G_S(\zeta;q)
	=:
	G_S(1;q)
	\exp\left( \sum_{n \geq 1} g_n(\tau) \frac{(2\pi i z)^n}{n!}\right),
	\qquad \zeta=e^{2\pi i z}.
\label{eq:intro_g_k}
\end{equation}
We also set \(	\mathcal{M}_r(q):={M_r(q)}\slash{G_S(1;q)}.\)
For convenience, we recall that, for a sequence \(\{X_j\}_{j\geq 1}\) and a partition
\(\lambda=(1^{m_1},\ldots,k^{m_k})\), we write
\[
	X_\lambda = \prod_{j=1}^k X_j^{m_j}.
\]
Then the definition of the complete Bell polynomials and the inversion formula
\eqref{eq:bell-inversion-int} give the following proposition.

\begin{prop} \label{prop:gen_stat_bell}
For \(r\geq 1\), we have
\begin{equation}
	\mathcal{M}_r (q)
	=
	B_r(g_1,g_2,\ldots,g_r) = \sum_{\lambda \vdash r} \phi_B(\lambda) g_\lambda.
\label{eq:Y_r_bell}
\end{equation}
Conversely, 
\[
	g_k(\tau)
	=
	\sum_{\lambda\vdash k}
	\mu(\lambda)\phi_B(\lambda) \mathcal{M}_\lambda .
\]
\end{prop}

\subsection{The largest part statistic}

Let \(L(\lambda)=\lambda_1\) denote the largest part of a partition
\(\lambda\), and let
\[
S_L(m,n):=\#\{\lambda\vdash n : L(\lambda)=m\}.
\]
Then
\[
P_L(\zeta;q)
:=
\sum_{m,n\geq 0}S_L(m,n)\zeta^m q^n
=
\sum_{m\geq 0}
\frac{\zeta^m q^m}{ (q)_m }
=
\frac{1}{  (\zeta q)_\infty  }.
\]

For this statistic $L(\lambda)$, we can explicitly evaluate $g_k(\tau)$ in \eqref{eq:intro_g_k}.
Multiplying by \((q)_\infty\) and taking logarithms, we obtain
\begin{align*}
\log\left((q)_\infty P_L(\zeta;q)\right)
&=
\sum_{j\geq 1}
\left(\log(1-q^j)-\log(1-\zeta q^j)\right) \\
&=
\sum_{j\geq 1}\sum_{n\geq 1}
\frac{\zeta^n-1}{n}q^{jn}  \\
&=
\sum_{n\geq 1}
\frac{\zeta^n-1}{n}\frac{q^n}{1-q^n}.
\end{align*}
Since \(\zeta=e^{2\pi iz}\), this becomes
\[
\log\left((q)_\infty P_L(\zeta;q)\right)
=
\sum_{r\geq 1}
D_r(q)\frac{(2\pi iz)^r}{r!},
\]
where
\begin{equation}
D_r(q)
:=
\sum_{n\geq 1}n^{r-1}\frac{q^n}{1-q^n}
=
\sum_{m\geq 1}\sigma_{r-1}(m)q^m.
\label{eq:divisor_power_sum}
\end{equation}
We use the notation \(D_r(q)\) to distinguish these divisor-sum series, which
occur for all \(r\geq 1\), from the Eisenstein series \(G_k(q)\) introduced
earlier.
In summary, we arrive at
\begin{equation}
P_L(\zeta;q)
=
\frac{1}{(q)_\infty}
\exp\left(
\sum_{r\geq 1}
D_r(q)\frac{(2\pi iz)^r}{r!}
\right).
\label{eq:largest-exp}
\end{equation}

We define the \(r\)-th moment generating function of the largest part by
\[
	P_{L,r}(q)
:=
\sum_{n\geq 0}\sum_{\lambda\vdash n}L(\lambda)^r q^n.
\]

Then Proposition~\ref{prop:gen_stat_bell} gives the following corollary.

\begin{cor}
Let $\mathcal{P}_{L,r}(q) := (q)_\infty P_{L,r} (q)$. Then, for \(r\geq 1\),
\begin{align*}
	\mathcal{P}_{L,r} (q)
	&= \sum_{\lambda \vdash r} \phi_B(\lambda) D_\lambda .  \\
\intertext{Conversely, for \(k\geq 1\),}
	D_k(q)
	&=
	\sum_{\lambda\vdash k}
	\mu(\lambda)\phi_B(\lambda) \mathcal{P}_{L,\lambda}(q),
\end{align*}
where
\[
	\mathcal{P}_{L,\lambda}(q)
	:=
	\prod_{j=1}^k \mathcal{P}_{L,j}(q)^{m_j}.
\]
\end{cor}

\section{Symmetric Combinatorial Statstics} \label{sec:main}

In this section, we consider integer-valued statistics \(s(\lambda)\) that are
symmetric, such as the partition crank and rank.

\subsection{Proof of Theorem~\ref{thm:main}}

We first prove our main theorem by employing the complete Bell polynomials and their inversion formula.

\begin{proof}[First proof of Theorem \ref{thm:main}]
By the classical expansion of \(\log(\sin z/z)\) (see, for example,
\cite[4.3.71]{AS1972}), we have
\[
\log\left(\frac{\sin z}{z}\right)
=
\sum_{n\geq 2}
\frac{B_n}{n\,n!}(2iz)^n.
\]
Recall that  $\mathcal{M}_r (q) = {M_r (q)}/{G_S (1;q)}$. We put
\[
A_1:=0,
\qquad
A_j:=2H_j(\tau)+\frac{B_j}{j}
\quad (j\geq 2).
\]
Then \eqref{eq:intro-H-exp} gives
\[
	G_S(\zeta;q)
	=
	G_S (1;q)
	\exp\left(
	\sum_{j\geq 1}
	A_j\frac{(2\pi iz)^j}{j!}
	\right).
\]
From the defining generating function of complete Bell polynomials, we see that
\[
	G_S (\zeta;q)
	=
	G_S(1;q)
	\sum_{r\geq 0}
	B_r(A_1,\ldots,A_r)
	\frac{(2\pi iz)^r}{r!}.
\]
Comparing the coefficients of \((2\pi iz)^r/r!\) in this expansion and in
\eqref{eq:stat-moment}, after dividing by \(G_S(1;q)\), we obtain
\begin{equation} \label{eq:first_M_r}
	\mathcal{M}_r(q)
	=
	B_r\left( 
	A_1,\,
	A_2,\,
	A_3,\,
	\dots,\,
	A_r
	\right).
\end{equation}

Conversely, the functions \(A_k(\tau)\) can be recovered from the
moment generating functions. By
\eqref{eq:bell-inversion-int}, for \(k\geq 1\), we have
\begin{equation} \label{eq:A_k_sym}
	A_k(\tau)
	=
	\sum_{\lambda\vdash k}
	\mu(\lambda)
	\phi_B(\lambda)
	\prod_{j\geq 1}
	\mathcal{M}_j(q)^{m_j}.
\end{equation}
Since the statistic is symmetric, we have \(\mathcal{M}_k(q)=0\) for odd
\(k\). Thus, it is immediate that $A_k (\tau) = 0$ for odd $k$ from \eqref{eq:A_k_sym}, because there is at least one odd part in a partition of an odd integer. Therefore, from \eqref{eq:first_M_r} and the bijection between the set of partitions of $k$ and the set of partitions of $2k$ with only even parts, we find that
\begin{align*}
	\mathcal{M}_{2k} (q) 
	&= \sum_{\lambda \vdash 2k} \phi_B (\lambda) A_\lambda \\
	&= \sum_{\lambda \vdash k} (2k)!\prod_{j=1}^k
	\frac{2^{m_j} }{m_j! ((2j)!)^{m_j}} \prod_{j=1}^{k} \left( \frac{A_{2j}}{2} \right)^{m_j}\\
	&=  \mathrm{Tr}_{k}\left( \phi_{M}; H_2 + \frac{B_2}{4}, H_4 + \frac{B_4}{8}, \ldots, H_{2k} + \frac{B_{2k}}{4k} \right),
\end{align*}
where $\phi_M (\lambda)$ is as defined by \eqref{eq:phi_M}.

Applying the same argument to the inversion formula \eqref{eq:A_k_sym}, only
partitions of \(2k\) with even parts contribute. Hence
\begin{align*}
	A_{2k}
	&=
	\sum_{\lambda\vdash 2k}
	\mu(\lambda)
	\phi_B(\lambda)
	\mathcal{M}_\lambda(q) \\
	&=
	\mathrm{Tr}_{k}\left(
	2\phi_{S};
	\mathcal{M}_2(q),
	\mathcal{M}_4(q),
	\ldots,
	\mathcal{M}_{2k}(q)
	\right),
\end{align*}
where \(\phi_S(\lambda)\) is defined in \eqref{eq:phi_S}. Dividing by \(2\) and using the definition of \(A_{2k}\), we obtain the second
identity of the theorem.
\end{proof}


We now give another proof of Theorem~\ref{thm:main} via Fa\`a di Bruno's formula.
		
\begin{proof}[Second proof of Theorem~\ref{thm:main}]	
We start with
\[
\log \frac{G_S(\zeta;q)}{G_S(1;q)} 
=
\sum_{j\geq 1}
A_j\frac{(2\pi iz)^j}{j!}. 
\]
Differentiating \(k\) times with respect to \(z\), we find
\begin{equation}
	\left(\frac{d}{dz}\right)^{k} \log \frac{G_S(\zeta;q)}{G_S(1;q)}
	=
	\sum_{n\ge k}\frac{(2\pi i)^n A_n(\tau)}{(n-k)!}z^{n-k}.
	\label{eq:logG-derivative}
\end{equation}
Evaluating at \(z=0\), we obtain
\begin{equation}
	L_k^G (q)
	:=
	\left.\left(\frac{d}{dz}\right)^k\log \frac{ G_S(\zeta;q)}{G_S(1;q)} \right|_{z=0}
	= (2\pi i)^k A_k (\tau).
\label{eq:der-log-G}
\end{equation}
We apply \eqref{eq:Faa-exp-log} to
\[
h(z):=\frac{G_S(e^{2\pi iz};q)}{G_S(1;q)}.
\]
Since \(h(0)=1\), we obtain
\begin{equation}
	(2\pi i)^k \mathcal{M}_k(q)
	=
	\mathrm{Tr}_k\bigl(
	\phi_B;
	L_1^G(q),L_2^G(q),\ldots,L_k^G(q)
	\bigr).
	\label{eq:faa-GS}
\end{equation}
Using \eqref{eq:der-log-G}, this gives
\[
	(2\pi i)^k \mathcal{M}_k(q)
	=
	\mathrm{Tr}_k\bigl(
	\phi_B;
	(2\pi i)A_1,(2\pi i)^2A_2,\ldots,(2\pi i)^kA_k
	\bigr).
\]
After cancelling the common factor \((2\pi i)^k\), we recover
\[
	\mathcal{M}_k(q)
	=
	B_k(A_1,\ldots,A_k).
\]
The first identity of the theorem follows from this identity exactly as in the
first proof.

Conversely, applying Fa\`a di Bruno's formula \eqref{eq:Faa} with $f(x)=\log x$ and $g(z)= G_S(\zeta;q)/G_S(1;q)$, we obtain
\begin{equation*}
\begin{aligned}
	&\left(\frac{d}{dz}\right)^{2k} \log \left( \frac{G_S(\zeta;q)}{G_S(1;q)} \right)  \\
	&=
	\mathrm{Tr}_{2k}\left(
	\phi_F;
	\left(\frac{d}{dz}\right) \frac{G_S(\zeta;q)}{G_S(1;q)},
	\left(\frac{d}{dz}\right)^2 \frac{G_S(\zeta;q)}{G_S(1;q)} ,
	\ldots,
	\left(\frac{d}{dz}\right)^{2k} \frac{G_S(\zeta;q)}{G_S(1;q)}
	\right),
\end{aligned}
	\label{eq:logG-FDB}
\end{equation*}
where, for \(\mu =(1^{m_1},\ldots,(2k)^{m_{2k}}) \vdash 2k\),
\[
\phi_F(\mu)
=
(2k)!
(-1)^{\ell(\mu)-1}(\ell(\mu)-1)!
\left(\frac{G_S(1;q)}{G_S(\zeta;q)}\right)^{\ell(\mu)}
\prod_{j=1}^{2k}\frac{1}{m_j!(j!)^{m_j}}.
\]
Evaluating at \(z=0\), and using
\[
\left.
\left(\frac{d}{dz}\right)^j
\frac{G_S(\zeta;q)}{G_S(1;q)}
\right|_{z=0}
=
(2\pi i)^j\mathcal{M}_j(q),
\]
we obtain
\begin{equation}
	\mathrm{Tr}_{2k}\left(
	\phi_F\big|_{z=0};
	(2\pi i) \mathcal{M}_1(q),(2\pi i)^2 \mathcal{M}_2(q),\ldots,(2\pi i)^{2k} \mathcal{M}_{2k}(q)
	\right)
	=
	 (2\pi i)^{2k} A_{2k}(\tau),
	\label{eq:logG-eval}
\end{equation}
where the right-hand side follows from \eqref{eq:der-log-G} with $k$ replaced by $2k$. Since \(\mathcal{M}_{2j-1}(q)=0\), only partitions of \(2k\) into even parts
contribute. As before, replacing each part \(2j\) by \(j\) reduces the sum to
one over partitions of \(k\). Therefore \eqref{eq:logG-eval} becomes
\[
(2\pi i)^{2k}
\mathrm{Tr}_k\left(
\widetilde{\phi}_{F};
\mathcal{M}_2(q), \mathcal{M}_4(q),\ldots, \mathcal{M}_{2k}(q)
\right)
=
(2\pi i)^{2k} A_{2k} (\tau)
\]
where, for \(\lambda=(1^{m_1},\ldots,k^{m_k})\vdash k\),
\[
\widetilde{\phi}_{F}(\lambda)
=
(2k)!
(-1)^{\ell(\lambda)-1}(\ell(\lambda)-1)! 
\prod_{j=1}^k \frac{1}{m_j!((2j)!)^{m_j}}.
\]
Cancelling $(2\pi i)^{2k}$ from both sides yields
\[
A_{2k}(\tau)
=
\mathrm{Tr}_k\left( \widetilde{\phi}_{F};
\mathcal{M}_2(q), \mathcal{M}_4(q),\ldots, \mathcal{M}_{2k}(q)
\right).
\]
Since \(\widetilde{\phi}_{F}=2\phi_S\),  the second identity of the
theorem follows from the definition of \(A_{2k}\). This completes the second
proof.
\end{proof}

\subsection{Partition crank}

We first note that
\begin{align}
	\frac{d}{dz}\log(C(\zeta; q))
	&=
	2\pi i
	\left(
	\sum_{n\geq 1}\frac{\zeta q^n}{1-\zeta q^n}
	-\sum_{n\geq 1}\frac{\zeta^{-1}q^n}{1-\zeta^{-1}q^n}
	\right)\label{infty Lambert series}\\
	&=
	2\pi i
	\left(
	\sum_{n,m\geq 1}(\zeta q^n)^m
	-\sum_{n,m\geq 1}(\zeta^{-1}q^n)^m
	\right).
	\nonumber
\end{align}
Iterating this differentiation yields
\begin{equation}
	\left.\left(\frac{d}{dz}\right)^k\log(C(\zeta; q))\right|_{z=0}
	=
	\begin{cases}
		0, & \text{if $k$ is odd},\\[4pt]
		2(2\pi i)^k  D_{k}, & \text{if $k$ is even},
	\end{cases}
\label{eq:der-log-crank}
\end{equation}
where $D_k (q)$ is the generating function for sums of divisor powers defined in \ \eqref{eq:divisor_power_sum}. Since
\(
	D_{2k}(q)=G_{2k}(\tau)+\frac{B_{2k}}{4k},
\)
comparison with \eqref{eq:der-log-G} shows that, in the crank case,
\[
H_{2k}(\tau)=G_{2k}(\tau).
\]
Therefore Theorem~\ref{thm:main} gives the following identities for the crank
moments.

\begin{cor}\label{cor:ptn_crank}
For every integer \(k\ge 1\), let $\mathcal{C}_k (q) = (q)_\infty C_k (q)$.  Then we have
\begin{align}
	\mathcal{C}_{2k}(q)
	&=
	\mathrm{Tr}_{k}
	\left(
	\phi_M;
	G_2+\frac{B_2}{4},
	G_4+\frac{B_4}{8},
	\ldots,
	G_{2k}+\frac{B_{2k}}{4k}
	\right)
	\label{eq:crank-trace-form}
	\\
\intertext{and}
	G_{2k}(\tau)
	&=
	-\frac{B_{2k}}{4k}
	+
	\mathrm{Tr}_{k}\left(
	\phi_{S};
	\mathcal{C}_2(q), \mathcal{C}_4(q), \ldots, \mathcal{C}_{2k}(q)
	\right).
	\label{eq:tr-c-g}
\end{align}
\end{cor}

\begin{rmk}
The corollary shows that the two families
\[
	\{\mathcal{C}_{2k}(q):k\geq 1\}
	\qquad\text{and}\qquad
	\{G_{2k}(\tau):k\geq 1\}
\]
generate the same \(\mathbb{Q}\)-algebra. Hence
\[
\mathbb{Q}[\mathcal{C}_2,\mathcal{C}_4,\ldots]
=
\mathbb{Q}[G_2,G_4,\ldots].
\]
\end{rmk}

\subsection{Partition rank}

The corresponding statement for the rank moments and mock Eisenstein series is as follows.

\begin{cor}\label{cor:ptn_rank}
For every integer \(k\ge 1\), let $\mathcal{R}_k (q) = (q)_\infty R_k (q)$. Then, we have
\begin{align}
	\mathcal{R}_{2k}(q)
	&=
	\mathrm{Tr}_{k}\left(
	\phi_M;
	f_2+\frac{B_2}{4},
	f_4+\frac{B_4}{8},
	\ldots,
	f_{2k}+\frac{B_{2k}}{4k}
	\right)
	\label{eq:rank-trace-form}
	\\
\intertext{and}
	f_{2k}(q)
	&=
	-\frac{B_{2k}}{4k}
	+
	\mathrm{Tr}_{k}\left(\phi_{S}; \mathcal{R}_2(q), \mathcal{R}_4(q), \ldots, \mathcal{R}_{2k}(q)\right).
	\label{eq:tr-r-f}
\end{align}
\end{cor}

\begin{rmk}
\begin{enumerate}
\item Since \(\mathcal{R}_{2k}(q)\) has no constant term in its \(q\)-expansion,
\eqref{eq:tr-r-f} immediately gives
\[
\lim_{\tau\to i\infty} f_{k}(\tau)=-\frac{B_{k}}{2k},
\]
as in \cite[Theorem~1.2(1)]{bringmannMock}. 
\item The corollary shows that
\[
	\mathbb{Q}[\mathcal{R}_2, \mathcal{R}_4,\ldots,G_2,G_4,\ldots]
	=
	\mathbb{Q}[f_2,f_4,\ldots,G_2,G_4, \ldots ].
\]
Therefore, as proved in \cite[Theorem~1.2(3)]{bringmannMock}, 
\[
	\mathbb{Q}[f_2,f_4,\ldots,G_2,G_4,\ldots ]
\]
is closed under the operator \(q\frac{d}{dq}\).
\end{enumerate}
\end{rmk}

We also obtain an alternative expression for partition rank moments by applying Fa\` a di Bruno's formula term by term. This gives a truncated divisor-sum analogue of the crank moment formula.

\begin{prop} \label{prop:finiteLambert}
For \(n, j \geq 1\), let
\[
	\sigma_{j}^{[n]}(m)
	:=
	\sum_{\substack{d\mid m\\ d\leq n}} d^{j}
\]
and
\[
	D_{j}^{[n]}(q)
	:=
	\sum_{m\geq 1}\sigma_{j-1}^{[n]}(m)q^m.
\]
Then, for every integer \(k\geq 1\), we have
\begin{equation} \label{eq:finite-Lambert-series}
	R_{2k}(q)
	=
	\sum_{n\geq 1}
	\frac{q^{n^2}}{(q)_n^2}
	\mathrm{Tr}_{k}\bigl(
	\phi_M;
	D_2^{[n]}, D_4^{[n]}, \ldots , D_{2k}^{[n]}
	\bigr).
\end{equation}
\end{prop}
\begin{proof}
We define
\[ 
	h_n(\zeta;q) := \frac{q^{n^2}}{(\zeta q)_n(\zeta^{-1}q)_n}, 
\]

so that $R(\zeta;q) = \sum_{n \geq 0} h_n$ with $h_0 = 1$. Taking the logarithmic derivative gives
\[
	\left(\frac{d}{dz}\right)^k \log h_n
	= (2\pi i)^k \sum_{i=1}^n \sum_{m=1}^\infty
	\left(m^{k-1}(\zeta q^i)^m - (-m)^{k-1}(\zeta^{-1}q^i)^m\right),
\]
which implies that 
\begin{equation*}
	L_k^{[n]} :=\left.\left(\frac{d}{dz}\right)^k \log h_n\right|_{z=0}
	=
	\begin{cases}
		0, & \text{if } k \text{ is odd},\\[4pt]
		2(2\pi i)^k D_{k}^{[n]} , & \text{if } k \text{ is even}.
	\end{cases}
\end{equation*}
Applying \eqref{eq:Faa-exp-log} with $h = h_n$, evaluating at \(z=0\),
and summing over $n \geq 1$, we obtain
\begin{equation*}
	(2\pi i)^{2k} R_{2k}(q)
	= \sum_{n \geq 1} \frac{q^{n^2}}{(q)_n^2}\,
	\mathrm{Tr}_{2k}\!\left(\phi_B; L_1^{[n]}, L_2^{[n]}, \ldots, L_{2k}^{[n]}\right).
\end{equation*}
Since \(L_{2j-1}^{[n]}=0\), only partitions of \(2k\) into even parts
contribute. Reducing to partitions of \(k\) and cancelling
\((2\pi i)^{2k}\) gives \eqref{eq:finite-Lambert-series}.
\end{proof}


\subsection{Unimodal rank moment}
A weakly unimodal sequence of weight \(n\) is a sequence
\[
	a_1 \le a_2 \le \cdots \le a_r \le \overline{c}
	\ge b_1 \ge b_2 \ge \cdots \ge b_s
\]
such that
\[
	\sum_{i=1}^{r} a_i+\sum_{j=1}^{s} b_j+c=n.
\]
Its rank is defined to be \(r-s\). Here the overline on \(c\) is used to
distinguish the peak from the other parts. Let \(u(m,n)\) be the number of
weakly unimodal sequences of weight \(n\) and rank \(m\). Then the two-variable
generating function is given by

\[
	U (\zeta, q) := \sum_{n \ge 0} \sum_{m \in \mathbb{Z}} u(m,n) \zeta^m q^n = \sum_{n \geq 0} \frac{q^n}{(\zeta q)_n (q/\zeta)_n}.
\]
The specialization \(U(1;q)\) satisfies 
\[
	U(1;q)
	=
	\sum_{n\geq 0}
	\frac{q^n}{(q)_n^2}
	=
	\frac{1}{(q)_\infty^2}
	\sum_{n\geq 0}(-1)^n q^{n(n+1)/2},
\]
where the last sum is a partial theta series \cite{War}.
There is also a Hecke--Appell type expression \cite[eq.~(2.5)]{KimLovejoy2014}
\[
	U(\zeta;q)
	=
	\frac{1-\zeta}{(q)_\infty^2}
	\left(
	\sum_{r,s\geq 0}
	-
	\sum_{r,s<0}
	\right)
	\frac{
	(-1)^{r+s}
	q^{\frac{r^2}{2}+2rs+\frac{s^2}{2}+\frac{3r}{2}+\frac{s}{2}}
	}
	{1-\zeta q^r}.
\]

We define the \(k\)-th unimodal rank moment generating function by
\[
	U_k(q)
	:=
	\left.
	\left(\zeta\frac{d}{d\zeta}\right)^k
	U(\zeta;q)
	\right|_{\zeta=1}.
\]
By symmetry, \(U_k(q)=0\) for odd \(k\). While the crank and rank moments of
ordinary partitions are closely related to Eisenstein series and mock
Eisenstein series, respectively, the rank moments of unimodal sequences are
related to certain Eisenstein-type series, called false Eisenstein series and
partial Eisenstein series introduced in \cite{bringmann2026false}. In that work, the
functions \(u_k(\tau)\) are defined by
\[
	U(\zeta;q)
	=
	\frac{\sin(\pi z)}{\pi z}U(1;q)
	\exp\left(
	2\sum_{k\geq 1}u_k(\tau)
	\frac{(2\pi iz)^k}{k!}
	\right),
	\qquad \zeta=e^{2\pi iz}.
\]

The modularity properties of \(u_k\) are given in
\cite[Corollary~1.8]{bringmann2026false}. As in the previous subsections,
Theorem~\ref{thm:main} gives explicit relations between the unimodal rank
moments and the functions \(u_{2k}\).

\begin{cor}\label{cor:unimodal_rank}
For every integer  \(k, j \geq 1\), we set
\begin{align*}
	\mathcal{U}_{2k}(q) &:=\frac{U_{2k}(q)}{U(1;q)} \\
	W_{2j}(\tau)
	&:=
	u_{2j}(\tau)+\frac{B_{2j}}{4j}.
\end{align*}
Then
\begin{align}
	\mathcal{U}_{2k}(q)
	&=
	\mathrm{Tr}_{k}
	\bigl(\phi_{M}; W_2,W_4,\ldots,W_{2k}\bigr)
	\label{eq:unimodal-rank-trace-form}
	\\
\intertext{and}
	W_{2k}(\tau)
	&=
	\mathrm{Tr}_{k}\left(
	\phi_{S};
	\mathcal{U}_2(q),\mathcal{U}_4(q),\ldots,\mathcal{U}_{2k}(q)
	\right).
	\label{eq:tr-U-u}
\end{align}
\end{cor}

\section{Applications} \label{sec:app}
We present some applications of the identities obtained above.

\subsection{Congruences}

By Corollary~\ref{cor:ptn_crank}, we have
\[
	C_2(q)
	=
	\frac{2}{(q)_\infty}D_2(q)
	=
	\frac{2}{(q)_\infty}
	\sum_{n\ge1}\sigma_1(n)q^n.
\]

On the other hand,
\[
q\frac{d}{dq}
\left(
\frac{1}{(q)_\infty}
\right)
=
\frac{1}{(q)_\infty}
\sum_{n\ge1}\sigma_1(n)q^n.
\]
Therefore
\[
	C_2(q)
	=
	2q\frac{d}{dq}
	\left(
	\frac{1}{(q)_\infty}
	\right),
\]
which implies that
\[
	c_2(n)=2np(n),
\]
where $C_2 (q) = \sum_{n \ge 1} c_2 (n) q^n$. This identity goes back to Dyson, who gave a combinatorial proof (see also
\cite[(1.27)]{AtkinGarvan2003}). It follows immediately that, for every prime \(p\),
\[
	c_2(pn)\equiv0\pmod p.
\]
Moreover, Ramanujan's congruences for \(p(n)\) imply
\begin{align*}
	c_2(5n+4) &\equiv0\pmod5, \\
	c_2(7n+5) &\equiv0\pmod7, \\
	c_2(11n+6)& \equiv0\pmod{11}.
\end{align*}
More generally, the existence theorems of Ono \cite{Ono2000} and Ahlgren--Ono \cite{AhlgrenOno2001}  for partition congruences
on arithmetic progressions yield corresponding existence results
for congruences of \(c_2(n)\).

We next use the inverse relation to recover the Eisenstein coefficient from
the second crank moment. From the preceding identity or  from
Corollary~\ref{cor:ptn_crank}, we have
\[
	\sum_{n\ge1}\sigma_1(n)q^n
	=
	\frac12 (q)_\infty C_2(q).
\]
By Euler's pentagonal number theorem, that is,
\[
	(q)_\infty
	=
	\sum_{r\in\mathbb Z}
	(-1)^r q^{r(3r-1)/2},
\]
we derive that
\[
	\sum_{n\ge1}\sigma_1(n)q^n
	=
	\frac12
	\left(
	\sum_{r\in\mathbb Z}
	(-1)^r q^{r(3r-1)/2}
	\right)
	\left(
	\sum_{m\ge0}c_2(m)q^m
	\right).
\]

Comparing coefficients gives
\[
\sigma_1(N)
=
\frac12
\sum_{r\in\mathbb Z}
(-1)^r
c_2\left(
N-\frac{r(3r-1)}2
\right),
\]
where we set \(c_2(m)=0\) for \(m<0\).

In particular, congruences for \(c_2(n)\) yield congruences for pentagonal
convolutions of the Eisenstein coefficients. For example, since
\[
c_2(5n+4)\equiv0\pmod5,
\]
we obtain
\[
\sigma_1(N)
\equiv
3
\sum_{\substack{r\in\mathbb Z\\
N-\frac{r(3r-1)}2\not\equiv 4\pmod5}}
(-1)^r
c_2\left(
N-\frac{r(3r-1)}2
\right)
\pmod5.
\]
Thus the inverse relation converts congruences for
crank moments into pentagonal-number convolution congruences for Eisenstein
coefficients.

The same argument applies to rank moments. In that case, the Fourier
coefficients of the mock Eisenstein series \(f_{2k} (\tau) \) can be expressed as
pentagonal-number convolutions involving rank moments. Higher moments can be
treated similarly, although the resulting formulas become increasingly
complicated.

We can also obtain a congruence application for unimodal rank moments. 
Kim and Lovejoy \cite{KimLovejoy2014} proved the second rank moment congruence
\[
	[q^{7n+6}]U_2 (q) \equiv 0 \pmod 7,
\]
where $[q^n] f(q)$ is the coefficient of $q^n$ in the $q$-expansion of $f(q)$. On the other hand, Corollary~\ref{cor:unimodal_rank} gives
\[
	U_2(q)=2 U(1;q) W_2(q),
	\qquad
	W_2(q)=u_2(\tau)+\frac{1}{24}.
\]
Therefore, if
\[
	U(1;q)=\sum_{n\geq0} u(n) q^n,
	\qquad
	W_2 (q)=\sum_{n\geq0}w_2 (n)q^n,
\]
then
\[
	\sum_{j=0}^{7n+6} u(j)w_2(7n+6-j)
	\equiv 0 \pmod 7.
\]
Thus the known congruence for the second unimodal rank moment gives a
modulo \(7\) convolution congruence involving the coefficient of the
corresponding false or partial Eisenstein series.


\subsection{Partition numbers and Bell polynomials}
We can express the partition number \(p(n)\) in terms of divisor sums by using
complete Bell polynomials, and conversely recover the divisor sums from the
partition numbers. We start with
\begin{equation}
q\frac{d}{dq}
\log\left( \frac{1}{(q)_\infty } \right)
=
\sum_{n\ge1}\sigma_1(n)q^n .
\label{eq:logdiffp}
\end{equation}
Integrating \eqref{eq:logdiffp} gives
\[
\sum_{n\ge0}p(n)q^n
=
\exp\left(
\sum_{n\ge1}\frac{\sigma_1(n)}{n}q^n
\right)
=
\exp\left(
\sum_{n\ge1}(n-1)!\sigma_1(n)\frac{q^n}{n!}
\right).
\]
Hence, by the defining relation for the complete Bell polynomials,
\[
p(n)
=
\frac{1}{n!}
B_n\bigl(
0!\sigma_1(1),
1!\sigma_1(2),
\ldots,
(n-1)!\sigma_1(n)
\bigr).
\]
Equivalently,
\[
p(n)
=
\sum_{\lambda=(1^{m_1},\ldots,n^{m_n})\vdash n}
\prod_{j=1}^n
\frac{1}{m_j!}
\left(\frac{\sigma_1(j)}{j}\right)^{m_j}.
\]
Conversely, taking the logarithm gives the inverse relation
\[
\sigma_1(n)
=
n
\sum_{\lambda=(1^{m_1},\ldots,n^{m_n})\vdash n}
\mu(\lambda)
\prod_{j=1}^n
\frac{p(j)^{m_j}}{m_j!}.
\]
As an illustration, the case $n=4$ yields:
\begin{align*}
5 = p(4) &= \frac{\sigma_1 (4)}{4} +  \frac{\sigma_1 (3)}{3} \cdot \frac{\sigma_1 (1)}{1} + \frac{1}{2!} \left( \frac{\sigma_1 (2)}{2} \right)^2 \\
&\qquad + \frac{\sigma_1(2)}{2} \cdot \frac{1}{2!} \left(\frac{\sigma_1(1)}{1} \right)^2+ \frac{1}{4!} \left( \frac{\sigma_1 (1)}{1} \right)^4, \\
\frac{7}{4} = \frac{\sigma_1 (4)}{4} &= \frac{p(4)}{1} - \frac{p(3)}{1}\cdot\frac{p(1)}{1} - \frac{p(2)^2}{2!} + 2! \frac{p(2)}{1}\cdot \frac{p(1)^2}{2!} - 3! \frac{p(1)^4}{4!}.
\end{align*}

\section{Concluding Remarks} \label{sec:con_rmk}
Our approach can also be applied to other combinatorial statistics. As a
possible further direction, we consider the rank for partitions into distinct
parts. Let \(q(n)\) denote the number of partitions of \(n\) into distinct parts. Its
generating function is
\[
Q(q):=\sum_{n\ge0} q(n)q^n
=
(-q)_\infty
=
q^{-1/24}\frac{\eta(2\tau)}{\eta(\tau)}
=
1+\sum_{n\ge1}\frac{q^{n(n+1)/2}}{(q)_n}.
\]
Thus up to $q^{-1/24}$, \(Q(q)\) is an eta-quotient of weight \(0\), and hence has modular transformation properties on a congruence subgroup.

For \(m\in\mathbb{Z}\) and \(n\ge0\), let \(Q(m,n)\) denote the number of partitions of \(n\) into distinct parts with rank \(m\). Then the generating function is
\begin{equation}
H(\zeta;q)
:=
\sum_{\substack{m\in\mathbb{Z}\\ n\ge0}} Q(m,n)\zeta^m q^n
=
\sum_{n\ge0}\frac{q^{n(n+1)/2}}{(\zeta q)_n}
=
1+\sum_{m\ge1}(\zeta q)^m(-\zeta^{-1}q)_{m-1}.
\label{eq:distinct-rank-gf}
\end{equation}

At \(\zeta=-1\), we have
\begin{equation}
H(-1;q)
=
1+\sum_{n\ge 1}\frac{q^{n(n+1)/2}}{(-q)_n}
=
\sigma(q),
\label{eq:distinct-rank-minus-one}
\end{equation}
which is Ramanujan's \(\sigma\)-function. 
Andrews, Dyson, and Hickerson \cite{AndrewsDysonHickerson1988} showed that infinitely many of the coefficients of \(\sigma(q)\) vanish, while the coefficients are unbounded, by relating them to a Hecke character of the real quadratic field \(\mathbb{Q}(\sqrt{6})\). 
Cohen \cite{Cohen1988} subsequently constructed a Maass waveform from \(\sigma(q)\) and its companion \(\sigma^*(q)\). 
These functions are now understood to be closely related to quantum modular forms. 
Zwegers \cite{Zwegers2012} later placed such examples into the broader framework of mock Maass theta functions, namely, certain non-modular eigenfunctions of the weight \(0\) hyperbolic Laplacian that can be completed to non-holomorphic modular forms of weight 0.

As before, define the distinct-rank moment generating functions by
\begin{equation}
H_k^{\mathrm{dist}}(q)
:=
\sum_{n\ge0}\sum_{m\in\mathbb Z} m^k Q(m,n)q^n
=
\left.
\left(\zeta\frac{d}{d\zeta}\right)^k H(\zeta;q)
\right|_{\zeta=1}.
\label{eq:distinct-rank-moment}
\end{equation}
We may also define functions \(h_k(\tau)\) by
\begin{equation}
H(\zeta;q)
=:
(-q)_\infty
\exp\left(
\sum_{n\geq 1}h_n(\tau)\frac{(2\pi iz)^n}{n!}
\right),
\qquad
\zeta=e^{2\pi iz}.
\label{eq:intro-distinct-h}
\end{equation}
Then Proposition~\ref{prop:gen_stat_bell} gives
\begin{align*}
	\mathcal{H}_r(q)
	:=
	\frac{H_r^{\mathrm{dist}}(q)}{(-q)_\infty}
	&=
	\sum_{\lambda\vdash r}\phi_B(\lambda)h_\lambda, \\
	h_k(\tau)
	&=
	\sum_{\lambda\vdash k}
	\mu(\lambda)\phi_B(\lambda)\mathcal{H}_\lambda.
\end{align*}
In view of the connection with Ramanujan's \(\sigma\)-function, it would be
interesting to understand whether the functions \(h_k(\tau)\) exhibit some
modularity-type behavior.

 It would also be interesting to use the results of this paper to develop
analogous constructions for other partition classes, such as overpartitions,
\(\ell\)-regular partitions, and multi-colored partitions. Another natural
direction is to apply the Bell-polynomial relations developed here to
symmetrized rank and crank moments, as well as to congruences for variants of
unimodal sequences studied in \cite{BringmannLovejoy2025}, \cite{BrysonOnoPitmanRhoades2012}, 
\cite{ChenGarvan2022}, and \cite{KimLimLovejoy2016}.

\section*{Acknowledgment}
Soon-Yi Kang and Chanhee Lee were supported by Basic Science Research Program through the National Research Foundation of Korea (NRF) funded by the Ministry of Education (RS-2025-25415913).
Byungchan Kim was supported by the Basic Science Research Program through the National Research Foundation of Korea (NRF) funded by the Ministry of Science and ICT (RS-2025-16065347).
Chanhee Lee was supported by Basic Science Research Program through the National Research Foundation of Korea (NRF) funded by the Ministry of Education (RS-2025-25433719).


\begin{thebibliography}{99}
\bibitem{AS1972}
M. Abramowitz and I. A. Stegun,
\emph{Handbook of Mathematical Functions with Formulas, Graphs, and Mathematical Tables},
Dover Publications, 1972.


\bibitem{AhlgrenOno2001}
S. Ahlgren and K. Ono, 
\emph{Congruence properties for the partition function},
Proc. Natl. Acad. Sci. USA {\bf 98} (2001), no. 23, 12882--12884.




\bibitem{TracepartitionOno2025}
T. Amdeberhan, M. Griffin, K. Ono, and A. Singh,
\emph{Traces of partition Eisenstein series},
Forum Math. {\bf 37} (2025), no.~6, 1417--1441.


\bibitem{AndrewsDysonHickerson1988}
G. E. Andrews, F. J. Dyson, and D. Hickerson,
\emph{Partitions and indefinite quadratic forms},
Invent. Math. {\bf 91} (1988), no. 3, 391--407.


\bibitem{AG1988}
G. E. Andrews and F. G. Garvan,
\emph{Dyson's crank of a partition},
Bull. Amer. Math. Soc. {\bf 18} (1988), no.~2, 167--171.

\bibitem{AtkinGarvan2003}
A. O. L. Atkin and F. G. Garvan,
\emph{Relations between the ranks and cranks of partitions},
in Number Theory and Modular Forms: Papers in Memory of Robert A. Rankin,
pp.~343--366, Kluwer Academic Publishers, 2003.

\bibitem{atkin-swinner}
A. O. L. Atkin and P. Swinnerton-Dyer,
\emph{Some properties of partitions},
Proc. London Math. Soc. {\bf 4} (1954), 84--106.


\bibitem{BringmannGarvanMahlburg2009}
K. Bringmann, F. G. Garvan, and K. Mahlburg,
\emph{Partition statistics and quasiharmonic Maass forms},
Inter. Math. Res. Notices, (2009), no.~1, 63--97.



\bibitem{BringmannLovejoy2025}
K. Bringmann and J. Lovejoy,
\emph{Odd unimodal sequences},
Adv. Math. {\bf 480} (2025), Paper No.~110436.


\bibitem{BringmannMahlburgRhoades2014}
K. Bringmann, K. Mahlburg, and R. C. Rhoades,
\emph{Taylor coefficients of mock-Jacobi forms and moments of partition statistics},
Math. Proc. Camb. Phil. Soc. {\bf 157} (2014), no.~2, 231--251.

\bibitem{bringmannMock}
K. Bringmann, B. V. Pandey, and J.-W. van Ittersum,
\emph{Eisenstein-type series associated to partition ranks},
arXiv preprint, arXiv:2504.07713, 2025.

\bibitem{bringmann2026false}
K. Bringmann, B. V. Pandey, and J.-W. van Ittersum,
\emph{False and partial Eisenstein series related to unimodal sequences},
arXiv preprint, arXiv:2601.19441, 2026.

\bibitem{BrysonOnoPitmanRhoades2012}
J. Bryson, K. Ono, S. Pitman, and R. C. Rhoades,
\emph{Unimodal sequences and quantum and mock modular forms},
Proc. Natl. Acad. Sci. USA {\bf 109} (2012), no.~40, 16063--16067.


\bibitem{ChenGarvan2022}
R. Chen and F. G. Garvan,
\emph{A proof of the mod 4 unimodal sequence conjectures and related mock theta functions},
Adv. Math. {\bf 398} (2022), Paper No.~108235.


\bibitem{Cohen1988}
H. Cohen,
\emph{\(q\)-identities for Maass waveforms},
Invent. Math. {\bf 91} (1988), no. 3, 409--422.

\bibitem{dyson}
F. J. Dyson,
\emph{Some guesses in the theory of partitions},
Eureka {\bf 8} (1944), 10--15.

\bibitem{Garvan}
F. G. Garvan,
\emph{New combinatorial interpretations of Ramanujan's partition congruences mod 5, 7 and 11},
Trans. Amer. Math. Soc. {\bf 305} (1988), no.~1, 47--77.

\bibitem{Kang2009}
S.-Y. Kang,
\emph{Mock Jacobi forms in basic hypergeometric series},
Compositio Math. {\bf 145} (2009), no.~3, 553--565.


\bibitem{KimLimLovejoy2016}
B. Kim, S. Lim, and J. Lovejoy,
\emph{Odd-balanced unimodal sequences and related functions: parity, mock modularity and quantum modularity},
Proc. Amer. Math. Soc. {\bf 144} (2016), no.~9, 3687--3700.

\bibitem{KimLovejoy2014}
B. Kim and J. Lovejoy,
\emph{The rank of a unimodal sequence and a partial theta identity of Ramanujan},
Int. J. Number Theory {\bf 10} (2014), no.~4, 1081--1098.

\bibitem{matsusaka2025applications}
T. Matsusaka,
\emph{Applications of Fa{\`a} di Bruno's formula to partition traces},
Res. Number Theory \textbf{11} (2025), no.~3, Article~69.


\bibitem{Ono2000}
K. Ono,
\emph{Distribution of the partition function modulo $m$},
Ann. of Math. (2) {\bf 151} (2000), no. 1, 293--307.


\bibitem{Rhoades2013}
R. C. Rhoades,
\emph{Families of quasimodular forms and Jacobi forms: The crank statistic for partitions},
Proc. Amer. Math. Soc. {\bf 141} (2013), no.~1, 29--42.


\bibitem{StanleyEC1}
R. P. Stanley,
\emph{Enumerative Combinatorics, Vol.~1},
2nd ed., Cambridge Univ. Press, 2011.

\bibitem{Speed1983}
T. P. Speed,
\emph{Cumulants and partition lattices},
Austral. J. Statist. {\bf 25} (1983), 378--388.




\bibitem{War}
S.~O.~Warnaar, \emph{Partial theta functions}, 
Srinivasa Ramanujan: His Life, Legacy, and Mathematical Influence, K. Alladi et al. (eds.), to appear.



\bibitem{Zwegers2012}
S. P. Zwegers,
\emph{Mock Maass theta functions},
Q. J. Math {\bf 63} (2012), 753–770.




\end{thebibliography}
\end{document}